\documentclass[letterpaper, 10pt, conference]{ieeeconf}

\usepackage{cite}
\usepackage{amsmath,amssymb,amsfonts}
\usepackage[scr = dutchcal]{mathalpha}
\usepackage{graphicx}
\usepackage{textcomp}
\usepackage{stmaryrd}
\usepackage{tabularx}
\usepackage{multirow}
\usepackage{float}
\usepackage{mathtools}
\usepackage{comment}

\usepackage{color}

\newcommand{\mbf}[1]{\mathbf{ #1}}
\newcommand{\tnf}[1]{\textnormal{#1}}
\newcommand{\tbf}[1]{\textbf{#1}}
\newcommand{\mbs}[1]{\boldsymbol{#1}}

\newcommand{\mcl}[1]{\mathcal{#1}}

\newcommand{\R}{\mathbb{R}}
\newcommand{\N}{\mathbb{N}}

\newcommand{\norm}[1]{\left\lVert{#1}\right\rVert}
\newcommand{\ip}[2]{\left\langle #1, #2 \right\rangle}

\newcommand{\bmat}[1]{\begin{bmatrix}#1\end{bmatrix}}

\newcommand{\mat}[1]{\begin{matrix}#1\end{matrix}}

\newcommand{\slbmat}[1]{{\small\left[\!\begin{array}{l}#1\end{array}\!\right]}}

\newcommand{\slmat}[1]{{\small\begin{array}{l}#1\end{array}}}

\newcommand{\sbrray}[2]{{{\scriptsize\left[\!\!\begin{array}{#1}#2\end{array}\!\!\right]}}}

\newtheorem{thm}{Theorem}
\newtheorem{defn}[thm]{Definition}
\newtheorem{lem}[thm]{Lemma}

\newtheorem{cor}[thm]{Corollary}

\floatstyle{ruled}
\newfloat{block}{thbp}{lop}
\floatname{block}{Block}

\let\bl\bigl
\let\bbl\Bigl
\let\bbbl\biggl

\let\br\bigr
\let\bbr\Bigr
\let\bbbr\biggr

\floatstyle{ruled}
\newfloat{block}{thbp}{lop}
\floatname{block}{Block}

\let\bl\bigl
\let\bbl\Bigl
\let\bbbl\biggl

\let\br\bigr
\let\bbr\Bigr
\let\bbbr\biggr

\usepackage{upgreek}

\newcommand\blfootnote[1]{%
	\begingroup
	\renewcommand\thefootnote{}\footnote{#1}%
	\addtocounter{footnote}{-1}%
	\endgroup
}

\title{\LARGE \bf
Verifying Well-Posedness of Linear PDEs using Convex Optimization
}

\author{Declan S. Jagt, Matthew M. Peet %
}

\begin{document}

	\maketitle
	\pagestyle{plain}

\begin{abstract}
Ensuring that a PDE model is well-posed is a necessary precursor to any form of analysis, control, or numerical simulation. Although the Lumer--Phillips theorem provides necessary and sufficient conditions for well-posedness of dissipative PDEs, these conditions must hold only on the domain of the PDE---a proper subspace of $L_{2}$---which can make them difficult to verify in practice. In this paper, we show how the Lumer--Phillips conditions for PDEs can be tested more conveniently using the equivalent Partial Integral Equation (PIE) representation. This representation introduces a fundamental state in the Hilbert space $L_{2}$ and provides a bijection between this state space and the PDE domain. Using this bijection, we reformulate the Lumer--Phillips conditions as operator inequalities on $L_{2}$. We show how these inequalities can be tested using convex optimization methods, establishing an upper bound on the exponential growth rate of solutions. We demonstrate the effectiveness of the proposed approach by verifying well-posedness for several classical examples of parabolic and hyperbolic PDEs.
\end{abstract}

\blfootnote{\vspace*{-0.00cm}%
	\tbf{Acknowledgement:} This work was supported by National Science Foundation grants 2337751 and 2429973. \vspace*{-0.25cm}}


\section{INTRODUCTION} 

Well-posedness is a fundamental property of any dynamic model, ensuring existence and regularity of solutions. For Ordinary Differential Equations (ODEs), the problem is relatively simple. Any linear ODE is well-posed, and even for nonlinear ODEs, existence of a non-increasing Lyapunov function is sufficient to establish this property. Well-posedness of Partial Differential Equations (PDEs), however, is more nuanced. Specifically, uniqueness of solutions necessitates the imposition of boundary conditions (and associated spatial Sobolev regularity). The well-posedness condition, then, requires that the Boundary Conditions (BCs) are sufficiently restrictive to avoid multiple solutions, while not being so tight as to preclude any solution. This process of ensuring that the BCs are compatible with the PDE is often difficult, requiring careful ad hoc analysis and integration of domain (BCs) and generator (PDE). In this paper, we propose a convex optimization algorithm for testing well-posedness, which applies to a large class of linear PDEs and BCs, while avoiding the need for extensive ad hoc analysis.

Let us consider a linear PDE of the form
\begin{equation}\label{eq:PDE_intro}
	\dot{\mbf{u}}(t)=A\mbf{u}(t),\quad t\geq 0,\qquad \mbf{u}(0)=\mbf{x},
\end{equation}
where $\mbf{u}(t)\in D$ denotes the state of the system, defined on a Sobolev subspace $D\subseteq H$ of a Hilbert space $H$, and where $A:D\to H$ is a linear differential operator. Many physical phenomena can be modeled by dynamical systems of this form, enabling analysis of system properties such as stability and facilitating the synthesis of controllers and observers.


A PDE model of the form in~\eqref{eq:PDE_intro} is well-posed if it admits a unique solution that depends continuously on the initial state, $\mbf{x}$~\cite{hadamard1902WellPosed}. A common framework for formulating well-posedness of PDEs is the theory of $C_{0}$-semigroups~\cite{curtain1995introduction}. These semigroups generalize the matrix exponential to infinite-dimensional settings, allowing solutions of~\eqref{eq:PDE_intro} to be characterized by the $C_{0}$-semigroup generated by $A$. Within this framework, the question of well-posedness is whether the operator $A$ generates a $C_{0}$-semigroup, and substantial research has been devoted to establishing generation theorems which provide conditions under which such a semigroup exists.

While necessary and sufficient conditions for generation of $C_{0}$-semigroups have been established in the form of the Hille–Yosida theorem\footnote{Also known as the Feller–Miyadera–Phillips theorem, for the more general formulation derived in~\cite{feller1952semigroup,miyadera1952generation,phillips1953semigroup}.}~\cite{hille1948functionalAnalysis,yosida1948semigroups}, verification of these conditions is challenging. Consequently, most well-posedness results use the stronger conditions proposed in the Lumer–Phillips theorem~\cite{lumerphillips1961dissipative}. These conditions are more closely related to the well-posedness test for nonlinear ODEs based on existence of a non-increasing Lyapunov function. Specifically, the Lumer--Phillips theorem states that $A$ generates a contraction semigroup (implying that the norm of solutions in the Hilbert space $H$ does not increase with time) if and only if $A$ satisfies $\ip{\mbf{x}}{A\mbf{x}}_{H}\leq 0$ for all $\mbf{x}\in D$, and $\lambda I-A:D\rightarrow H$ is surjective for some $\lambda>0$. Modifications of the Lumer–Phillips theorem have also been proposed for other classes of systems, including time-varying PDEs~\cite{schnaubelt2010TimeVarying,kurula2019TimeVarying} and partial differential–algebraic equations~\cite{sviridyuk2012PDAE,jacob2022solvabilityPDAE,gernandt2025DAEs}.

Although the Lumer--Phillips theorem is widely used to verify well-posedness of PDEs, application of this theorem still involves substantial analysis. For example, dissipativity is commonly verified using techniques such as integration by parts and the Cauchy--Schwarz inequality to establish that $\ip{\mbf{x}}{A\mbf{x}}_{H}\leq 0$ for all $\mbf{x}\in D$. Similarly, surjectivity of $\lambda I-A$ is frequently tested by explicitly inverting the operator $\lambda I-A$. Such arguments must generally be developed on a case-by-case basis, underscoring the need for a systematic, computational framework for verifying well-posedness. Developing such a framework, however, is complicated by the dependence of the operator $A:D\to H$ on the PDE domain $D$, for which no universal parameterization is available.

To address this challenge with developing a computational framework for well-posedness analysis of PDEs, we consider the Partial Integral Equation (PIE) representation. This PIE representation provides an equivalent state-space formulation of PDEs by establishing a bijection between the PDE domain and the Hilbert space $L_{2}$. Under suitable rank conditions on the BCs defining $D$ (see~\cite[Sec XIV.3, Thm 3.1]{gohberg1990LinearOperators}), such a bijection exists and has the form
\begin{equation*}
	(\mcl{T}\mbf{v})(s):=\int_{0}^{s}\mbs{T}_{1}(s,\theta)\mbf{v}(\theta)\, d\theta
	+\int_{s}^{1}\mbs{T}_{2}(s,\theta)\mbf{v}(\theta)\, d\theta,
\end{equation*}
where the kernels $\mbs{T}_{1}$ and $\mbs{T}_{2}$ are polynomial and can be computed explicitly from the BCs defining $D$. We refer to operators of this type as Partial Integral (PI) operators\footnote{PI operators are formally defined in Subsec.~\ref{subsec:LPI:LPI} and may be augmented with multiplier operators. Both with and without multipliers, the set of PI operators forms a *-algebra under composition, with analytic formulae for algebraic operations expressed using the associated kernels and multipliers.}. Using this bijection, $\mcl{T}:L_{2}\to D$, an equivalent representation of the PDE in~\eqref{eq:PDE_intro} can be established as a PIE,
\begin{equation*}
	\partial_{t}\mcl{T}\mbf{v}(t)=\mcl{A}\mbf{v}(t),\quad t\geq 0,
	\qquad \mbf{v}(0)=\mcl{T}^{-1}\mbf{x},
\end{equation*}
where $\mcl{A}:=A\circ\mcl{T}$ is also a PI operator. It has been shown that a broad class of linear ODE--PDE systems can be represented equivalently as PIEs of this form~\cite{shivakumar2024GPDE,jagt2025NDPIE_arXiv}, where $\mbf{u}$ satisfies the PDE if and only if $\mbf{v}=\mcl{T}^{-1}\mbf{u}$ satisfies the PIE.

Within the PIE framework, various analysis and control problems can be formulated as convex optimization problems over PI operator variables. Such optimization problems are termed Linear PI Inequalities (LPIs). Since PI operators may be parameterized by their polynomial kernels and multipliers, these LPIs can be solved using semidefinite programming methods through the PIETOOLS software suite~\cite{shivakumar2025PIETOOLS}.

Despite the growing use of the PIE representation for analysis and control of PDE systems, a corresponding generation theorem has not yet been formulated in the PIE framework. The challenge is that the operator $\mcl{T}$ on the left-hand side of the PIE dynamics prevents direct application of classical results such as the Lumer–Phillips theorem. In particular, while the condition $\ip{\mbf{u}}{A\mbf{u}}_{H}\leq 0$ can be readily tested in the PIE representation as an LPI, $\mcl{T}^*\mcl{A}+\mcl{A}^*\mcl{T}\preceq 0$, an analogous test for surjectivity of $\lambda I-A$ for $\lambda>0$ has not been developed. 
Although generation results do exist for descriptor systems, $\frac{d}{dt}E\mbf{x}(t)=A\mbf{x}(t)$~\cite{sviridyuk2012PDAE,jacob2022solvabilityPDAE,gernandt2025DAEs}, these results still assume $A$ to be defined on a domain $D$, leading to conditions that remain difficult to verify. This raises the challenge of using the bijective nature of $\mcl{T}$ and algebraic properties of $\{\mcl{T},\mcl{A}\}$ to formulate a simplified generation theorem, whose conditions can be tested numerically. 

In this paper, we address this challenge by deriving well-posedness conditions directly in the PIE representation and expressing them as LPIs. First, in Section~\ref{sec:PIEs}, we parameterize a class of 1D linear PDEs admitting an equivalent PIE representation, and pose well-posedness on $L_{2}$ in terms of the PIE. In Section~\ref{sec:LumerPhillips_PIE}, we then establish a Lumer--Phillips theorem for PIEs, yielding necessary and sufficient conditions for generation of a contraction semigroup. Finally, in Section~\ref{sec:LPI}, we express subsequent sufficient conditions for well-posedness in terms of an LPI, further establishing a bound on the exponential growth rate of solutions. We apply this test to several classical PDE examples in Section~\ref{sec:Examples}, using the PIETOOLS software to verify well-posedness and compute tight bounds on the growth rate of solutions.


\section{Preliminaries}

\subsection{Notation}\label{sec:preliminaries:notation}

$\R^{n \times m}[s]$ is the space of matrix-valued polynomials on $s$.
For $d,n\in\N$, denote by $L_{2}^{n}[a,b]$ the Hilbert space of $\R^{n}$-valued square-integrable functions on $[a,b]$ with the standard inner product $\ip{\cdot}{\cdot}_{L_{2}}$, and where we omit the domain when clear from context. Define the Sobolev subspace
\begin{equation*}
	W_{2}^{d,n}[a,b]:=\{\mbf{u}\in L_{2}^{n}\mid \partial_{s}^{k}\mbf{u}\in L_{2}^{n}[a,b],~\forall k\in\{1,\hdots,d\}\}.
\end{equation*}
For a real Hilbert space $H$, denote by $\mcl{L}(H)$ the space of bounded linear operators on $H$ (with operator norm induced by $\ip{\cdot}{\cdot}_{H}$), and let $I\in\mcl{L}(H)$ denote the identity operator. For coercive $\mcl{P}\in\mcl{L}(L_{2}^{n}[a,b])$, so that $\mcl{P}\succeq\epsilon I$ for some $\epsilon>0$, define the Hilbert space $L_{2,\mcl P}^{n}[a,b]$ to be $L_2^{n}[a,b]$, equipped with the weighted inner product $\ip{\mbf{x}}{\mbf{y}}_{L_{2,\mcl P}}:=\ip{\mbf{x}}{\mcl{P}\mbf{y}}_{L_{2}}$.
For $\mcl{P}\in\mcl{L}(H)$, denote by $\tnf{ran}(\mcl{P})$ and $\tnf{ker}(\mcl{P})$ the range and nullspace of $\mcl{P}$, respectively.

\subsection{Strongly Continuous Semigroups}

Well-posedness of solutions to linear PDEs as in~\eqref{eq:PDE_intro} will be verified using the theory of $C_{0}$-semigroups. For this, we recall the following definitions.
\begin{defn}[Infinitesimal Generator of $C_{0}$-Semigroup]
	For a Hilbert space $H$, a $\mbs{C_{0}}$\tbf{-semigroup} (or strongly continuous semigroup) on $H$ is a family $(S(t))_{t\geq 0}$ of operators in $\mcl{L}(H)$ with the following properties:
	\begin{enumerate}
		\item 
		$S(0)=I$;
		
		\item 
		$S(t+s)=S(t)S(s)$ for all $t,s\geq 0$;
		
		\item 
		$\lim_{t\to 0^{+}}S(t)\mbf{x}=\mbf{x}$ for all $\mbf{x}\in H$.
	\end{enumerate}
	For a $C_{0}$-semigroup $(S(t))_{t\geq 0}$, defining
	\begin{equation*}
		D:=\left\{\mbf{x}\in H\, \bbl|\, A\mbf{x}:=\lim_{t\to 0^{+}}\frac{1}{t}[S(t)-I]\mbf{x}\text{ exists in }H\right\},
	\end{equation*}
	we refer to $A:D\to H$ as the \tbf{infinitesimal generator} of $(S(t))_{t\geq 0}$. In this case, we write $e^{tA}:=S(t)$ for $t\geq0$.
\end{defn}

\section{Well-Posedness in the PIE Representation}\label{sec:PIEs}

The main contribution of this paper is a convex optimization program for verifying well-posedness of a class of linear PDEs. This optimization program will be established using the equivalent Partial Integral Equation (PIE) representation of the PDE, where if $D\subseteq W_{2}^{d,n}$, the PIE state is the $d^{\tnf{th}}$-order spatial derivative of the PDE state. In this section, we recall how this PIE representation may be constructed for a broad class of linear 1D PDEs. We then formulate the problem of well-posedness analysis in terms of the associated PIE representation.

\subsection{Bijection between PDE Domain and $L_2$}

We begin by defining a class of PDE domains compatible with the PIE representation.

\begin{defn}\label{defn:PDE_dom}
	For $n,d\in\N$, $[a,b]\subseteq\R$, and $B,C\in\R^{nd\times nd}$, define the associated domain $D\subseteq W_{2}^{d,n}[a,b]$ as
	\begin{equation*}
		D:=\bbbl\{\mbf{u}\in W_{2}^{d,n}[a,b]\,\bbbl|\
		\sum_{j=0}^{d-1}B_{i,j}\partial_{s}^{j}\mbf{u}(a)+C_{i,j}\partial_{s}^{j}\mbf{u}(b)=0\bbbr\},
	\end{equation*}
	where $B_{i,j},C_{i,j}\in\R^{n\times n}$ are concatenated as $B=[B_{i,j}]$ and $C=[C_{i,j}]$ for $i,j\in\{0,\hdots,d-1\}$.
	We say that $D$ is \tbf{PIE-compatible} if $K:=B+C\mbs{Q}(b-a)$ is invertible, where \\[-0.5\baselineskip]
	\begin{equation*}
		\mbs{Q}(z):={\small\bmat{I_{n}&zI_{n}&\cdots&\frac{z^{d-1}}{(d-1)!}I_{n}\\0_{n}&I_{n}&\ddots&\frac{z^{d-2}}{(d-2)!}I_{n}\\\vdots&\ddots&\ddots&\vdots\\0_{n}&0_{n}&\cdots&I_{n}}}.
	\end{equation*}
\end{defn}
\vspace*{2mm}

Most common boundary conditions (e.g., Dirichlet, Robin, mixed) yield PIE-compatible domains\footnote{The case of periodic boundary conditions is not prima facie PIE-compatible, but can be formulated in the PIE representation as in~\cite{jagt2025PIEperiodicBCs}.}. The following result, a corollary of~\cite[Sec~XIV.3,~Thm~3.1]{gohberg1990LinearOperators}, shows that if $D$ is PIE-compatible, then $\partial_{s}^{d}:D\to L_{2}^{n}[a,b]$ is invertible.

\begin{cor}\label{cor:Tmap}
	If $D\subseteq W_{2}^{d,n}[a,b]$ is PIE-compatible, then $\partial_{s}^{d}:D\to L_{2}^{n}[a,b]$ is invertible, and $\mcl{T}:=(\partial_{s}^{d})^{-1}:L_{2}^{n}[a,b]\to D$ takes the form $(\mcl{T}\mbf{v})(s):=\int_{a}^{b}\mbs{G}(s,\theta)\mbf{v}(\theta)\, d\theta,$
	where $\mbs{G}(s,\theta)=E_{1}^{T}\hat{\mbs{G}}(s,\theta)E_{d}$ for
	\begin{equation*}
		\hat{\mbs{G}}(s,\theta):=\begin{cases}
			\mbs{Q}(s-\theta)-\mbs{Q}(s-a)K^{-1}C\mbs{Q}(b-\theta),	&	\theta\leq s,\\
			-\mbs{Q}(s-a)K^{-1}C\mbs{Q}(b-\theta),	&	\theta>s,
		\end{cases}
	\end{equation*}
	for $\mbs{Q}(s)$, $K$, and $C$ as in Defn.~\ref{defn:PDE_dom}, and\\[-1\baselineskip]
	\begin{equation*}
		E_{1}:=\bmat{I_{n}&0_{n}&\!\cdots\!&0_{n}}^T,\quad
		E_{d}:=\bmat{0_{n}&\!\cdots\!&0_{n}&I_{n}}^T.
	\end{equation*}
\end{cor}
\vspace*{2mm}

Cor.~\ref{cor:Tmap} shows that if a PDE domain $D\subseteq W_{2}^{d,n}$ is PIE-compatible as in Defn.~\ref{defn:PDE_dom}, then for any $\mbf{v}\in L_{2}^{n}$, there exists a unique solution, $\mbf{u}$, to the boundary value problem
\begin{equation}\label{eq:BVP}
	\partial_{s}^{d}\mbf{u}=\mbf{v},\qquad \mbf{u}\in D.
\end{equation}
Moreover, this solution can be explicitly constructed as $\mbf{u}=\mcl{T}\mbf{v}$, where $\mcl{T}$ is a bounded, linear operator on $L_{2}^{n}$. In this manner, for a linear 1D PDE, PIE compatibility of the PDE domain already establishes well-posedness of the underlying boundary value problem in~\eqref{eq:BVP}. This will greatly simplify well-posedness analysis of the full PDE, by using the operator $\mcl{T}$ to establish a state-space representation of the PDE in terms of PIE state $\mbf{v}(t)\in L_{2}^{n}$.

\subsection{The PIE Representation of Linear 1D PDEs}

Given the class of PIE-compatible domains from the previous subsection, we now consider the challenge of verifying well-posedness of linear PDEs on such domains. In particular, consider the following abstract Cauchy problem,
\begin{equation}\label{eq:Cauchy}
	\dot{\mbf{u}}(t)=A\mbf{u}(t),\quad t\geq 0,\qquad
	\mbf{u}(0)=\mbf{x},
\end{equation}
on $D\subseteq L_{2}^{n}[a,b]$, where $\mbf{x}\in D$ is the initial state and where we define a class of linear operators
\begin{equation}\label{eq:Aop}
	(A\mbf{u})(s):=\sum_{k=0}^{d}\mbs{A}_{k}(s)\partial_{s}^{k}\mbf{u}(s),\quad \mbf{u}\in D,
\end{equation}
parameterized by $\mbs{A}_{k}\in \R^{n\times n}[s]$ for $k\in\{0,\hdots,d\}$. Note that this parameterization supports a broad class of linear, 1D PDEs, several examples of which are provided in Section~\ref{sec:Examples}. Given this class of PDEs, we will consider the following definition of well-posedness.

\begin{defn}[Well-Posed PDE]\label{defn:WellPosedPDE}
	For a given subspace $D\subseteq W_{2}^{d,n}[a,b]$ and $A:D\to L_{2}^{n}[a,b]$, we say that the PDE defined by $A$ is \tbf{well-posed} if $A$ is the infinitesimal generator of a $C_{0}$-semigroup on $L_{2}^{n}[a,b]$. 
\end{defn}

Note that well-posedness of a given PDE depends not only on the operator $A$, but also on the domain $D$ and the underlying Hilbert space on which the system is defined. In order to present a unified result, Defn.~\ref{defn:WellPosedPDE} considers only well-posedness on the Hilbert space $L_2^n$. This is significant for PDE models such as the wave equation in Subsection~\ref{subsec:Examples:Wave}---which is not prima facie well-posed on $L_2$.

To verify well-posedness of PDEs as in~\eqref{eq:Cauchy}, recall that if the PDE domain $D$ is PIE-compatible, then any element of $D$ may be uniquely identified with an element of $L_{2}^{n}$ using the bijection $\mcl{T}:=(\partial_{s}^{d})^{-1}:L_{2}^{n}\to D$. Introducing the state transformation $\mbf{u}(t)=\mcl{T}\mbf{v}(t)$, then, the PDE in~\eqref{eq:Cauchy} can be equivalently represented as a PIE, taking the form
\begin{equation}\label{eq:PIE_standard}
	\partial_{t}\mcl{T}\mbf{v}(t)=\mcl{A}\mbf{v}(t),\quad t\geq 0,\qquad \mbf{v}(0)=\mcl{T}^{-1}\mbf{x},
\end{equation}
where $\mcl{A}\in\mcl{L}(L_{2}^{n}[a,b])$ is defined by
\begin{equation}\label{eq:Aop_PIE}
	\mcl{A}:=A\circ\mcl{T}:=\sum_{k=0}^{d}\tnf{M}_{\mbs{A}_{k}}\circ\partial_{s}^{k}\circ\mcl{T},
\end{equation}
with $(\tnf{M}_{\mbs{A}_{k}}\mbf{v})(s):=\mbs{A}_{k}(s)\mbf{v}(s)$. It is not difficult to see that for any solution to the PDE, there exists an associated solution to the PIE, and vice versa, establishing equivalence of the two representations~\cite{shivakumar2024GPDE}. As such, a notion of well-posedness of the PIE can also be defined in terms of the associated PDE as follows.
\begin{defn}[Well-Posed PIE]
	For $\mcl{T}:L_{2}^{n}[a,b]\to D$ and $\mcl{A}\in\mcl{L}(L_{2}^{n}[a,b])$ bounded linear operators, with $\mcl{T}$ bijective, we say that the PIE defined by  $\{\mcl{T},\mcl{A}\}$ is \tbf{well-posed} if $A:=\mcl{A}\circ\mcl{T}^{-1}:D\to L_{2}^{n}$ is the infinitesimal generator of a $C_{0}$-semigroup. We refer to the $C_{0}$-semigroup generated by $\{\mcl{T},\mcl{A}\}$ as that generated by $A$.
\end{defn}

By this definition, well-posedness of a given PIE is equivalent to well-posedness of its associated PDE. The key advantage of the PIE formulation, however, is that both $\mcl{T}$ and $\mcl{A}$ are bounded operators on $L_{2}^n$, and the boundary conditions and regularity constraints are embedded in $\mcl{T}$. This will allow conditions for well-posedness of the PIE to be formulated in terms of operator inequalities and tested using convex optimization, as we show in the following sections.

\section{A Lumer--Phillips Theorem for PIEs}
\label{sec:LumerPhillips_PIE}

To establish well-posedness conditions in the PIE representation, we first develop a PIE-based formulation of the classical Lumer--Phillips generation theorem (see e.g.~\cite[Thm~3.4.2]{vrabie2003semigroups}). This result enables verification of contraction and quasicontraction semigroups, defined as follows.

\begin{defn}[Contraction Semigroup]
	For a Hilbert space $H$, a $C_{0}$-semigroup $(S(t))_{t\geq 0}$ on $H$ is a \tbf{quasicontraction semigroup} with rate $\omega\in\R$ if $\norm{S(t)}_{H}\leq e^{\omega t}$ for all $t\geq 0$. A \tbf{contraction semigroup} on $H$ is a quasicontraction semigroup with rate $\omega=0$.
\end{defn}

Generators of contraction semigroups correspond to PDEs for which the origin is Lyapunov stable, so that the functional $V(\mbf{u}):=\norm{\mbf{u}}_{H}^{2}$ is non-increasing along solutions. This requires the generator $A:D\to H$ to be \textit{dissipative}, i.e., $\ip{\mbf{u}}{A\mbf{u}}_{H}\leq 0$ for all $\mbf{u}\in D$---a condition that is sufficient for well-posedness of nonlinear ODEs. For PDEs, however, an additional condition is required to ensure that $e^{tA}$ is well-defined, as provided by the Lumer--Phillips theorem.

\begin{thm}[Lumer--Phillips Theorem]\label{thm:LP}
	For a given Hilbert space $H$ and subspace $D\subseteq H$, let $A:D\to H$ be a linear operator. Then $A$ is the infinitesimal generator of a contraction semigroup on $H$ if and only if $\ip{\mbf{u}}{A\mbf{u}}_{H}\leq 0$ for all $\mbf{u}\in D$ and the operator $\lambda I-A:D\to H$ is surjective for some (or equivalently, all) $\lambda >0$.
\end{thm}

Although the Lumer--Phillips theorem only characterizes generators of contraction semigroups, it is well known that $A$ generates a contraction semigroup if and only if $A+\omega I$ generates a quasicontraction semigroup with rate $\omega$~\cite{engel2000C0Semigroups}, allowing direct application of the theorem to this broader class of generators. 
However, testing the conditions of the Lumer--Phillips theorem---in particular the surjectivity of $\lambda I-A$ for $\lambda>0$---still requires substantial analysis.

In the following subsections, we reformulate the conditions of the Lumer--Phillips theorem in terms of the associated PIE, defined by $\{\mcl{T},\mcl{A}\}$, using the bijectivity of $\mcl{T}:L_{2}^{n}\to D$ to show that these conditions can be tested more conveniently in the PIE representation. In the subsequent subsection, we then establish a PIE-based Lumer--Phillips theorem, providing necessary and sufficient conditions for $\{\mcl{T},\mcl{A}\}$ to generate a quasicontraction semigroup. These results are stated for a general Hilbert space $H$, and will be applied to $H=L_{2,\mcl{P}}^n$ in the next section, to verify generation of more general, non-quasicontraction  semigroups on $L_{2}^n$.

\subsection{Dissipativity in the PIE Representation}
\label{subsec:LumerPhillips_PIE:dissipative}

Consider first the dissipativity condition, $\ip{\mbf{u}}{A\mbf{u}}_{H}\leq 0$ for all $\mbf{u}\in D$. This condition is equivalent to Lyapunov stability of the origin for the PDE, for which conditions have already been established in the associated PIE representation, and may be formulated as in the following lemma.

\begin{lem}\label{lem:dissipative}
	For a Hilbert space $H$, let $\mcl{T}:H\to D$ and $\mcl{A}:H\to H$ be bounded linear operators, with $\mcl{T}$ bijective. Define $A:=\mcl{A}\circ\mcl{T}^{-1}:D\to H$. Then $\ip{\mbf{u}}{A\mbf{u}}_{H}\leq 0$ for all $\mbf{u}\in D$ if and only if $\ip{\mcl{T}\mbf{v}}{\mcl{A}\mbf{v}}_{H}\leq 0$ for all $\mbf{v}\in H$.
\end{lem}

\begin{proof}
	For sufficiency, suppose $\ip{\mcl{T}\mbf{v}}{\mcl{A}\mbf{v}}_{H}\leq 0$ for all $\mbf{v}\in H$. Then, for any $\mbf{u}\in D$, letting $\mbf{v}:=\mcl{T}^{-1}\mbf{u}\in H$,
	\begin{equation*}
		\ip{\mbf{u}}{A\mbf{u}}_{H}
		=\ip{\mcl{T}\mbf{v}}{(\mcl{A}\circ\mcl{T}^{-1})\mcl{T}\mbf{v}}_{H}
		=\ip{\mcl{T}\mbf{v}}{\mcl{A}\mbf{v}}_{H}\leq 0 .
	\end{equation*}
	Similarly, for necessity, suppose $\ip{\mbf{u}}{A\mbf{u}}_{H}\leq 0$ for all $\mbf{u}\in D$. Then, for any $\mbf{v}\in H$, letting $\mbf{u}:=\mcl{T}\mbf{v}\in D$,
	\begin{equation*}
		\!\!\ip{\mcl{T}\mbf{v}}{\mcl{A}\mbf{v}}_{H}
		=\ip{\mbf{u}}{(A\circ\mcl{T})\mcl{T}^{-1}\mbf{u}}_{H}
		=\ip{\mbf{u}}{A\mbf{u}}_{H}\leq 0.
	\end{equation*}%
	\ \\[-1.65\baselineskip]
\end{proof}

By Lem.~\ref{lem:dissipative}, dissipativity of $A=\mcl{A}\circ\mcl{T}^{-1}$ can be verified by testing $\ip{\mcl{T}\mbf{v}}{\mcl{A}\mbf{v}}_{H}\leq0$ for all $\mbf{v}\in H$. Here, while $A$ is defined only on $D$, the operators $\mcl{T}$ and $\mcl{A}$ act on all of $H$. 
This allows dissipativity on the $L_{2}$ space $H=L_{2,\mcl{P}}^n$ with inner product weighted by $\mcl{P}$ to be tested as an operator inequality on $\mcl{T}$, $\mcl{A}$, and $\mcl{P}$, as shown in Section~\ref{sec:LPI}.


\subsection{Surjectivity in the PIE Representation}

Consider now the second condition of the Lumer--Phillips theorem: surjectivity of $\lambda I-A:D\to H$ for some $\lambda>0$. 
Although verifying surjectivity of $\lambda I-A$ for general differential operators $A: D\to H$ is nontrivial, Cor.~\ref{cor:Tmap} shows that for $H=L_{2}^{n}$ and $D\subseteq W_{2}^{d,n}$ parameterized as in Defn.~\ref{defn:PDE_dom}, surjectivity of $\partial_{s}^{d}:D\to H$ is implied by PIE compatibility of $D$, with an explicit inverse given by $\mcl{T}$. Using this fact, surjectivity of $\lambda I-A=(\lambda\mcl{T}-\mcl{A})\mcl{T}^{-1}:D\to H$ is then equivalent to surjectivity of $(\lambda\mcl{T}-\mcl{A}):H\to H$, for any $\lambda>0$. Unlike $\lambda I-A$, however, $\lambda\mcl{T}-\mcl{A}$ is a bounded operator, allowing surjectivity to be analyzed using the open mapping theorem. In particular, the following lemma (a reformulation of~\cite[Thm~4.13]{rudin1991functionalAnalysis})
shows that this property is equivalent to $(\lambda\mcl{T}-\mcl{A})^*$ being bounded below.

\begin{lem}\label{lem:surjective}
	Let $\mcl{B}\in\mcl{L}(H)$ for a Hilbert space $H$. Then $\mcl{B}$ is surjective if and only if there exists $\epsilon>0$ such that\\[-0.75\baselineskip]
	\begin{equation*}
		\norm{\mcl{B}^*\mbf{x}}_{H}\geq\epsilon\norm{\mbf{x}}_{H},
		\qquad \forall \mbf{x}\in H.
	\end{equation*}
\end{lem}
\vspace*{1mm}
\begin{proof}
	For sufficiency, suppose there exists $\epsilon>0$ such that $\norm{\mcl{B}^*\mbf{x}}_{H}\geq\epsilon\norm{\mbf{x}}_{H}$ for all $\mbf{x}\in H$. Then $\norm{\mcl{B}^*\mbf{x}}_{H}=0$ implies $\mbf{x}=0$, whence $\tnf{ker}(\mcl{B}^*)=\{0\}$. By the closed range theorem, in order to prove $\tnf{ran}(\mcl{B})=\tnf{ker}(\mcl{B}^*)^{\perp}=H$, it suffices to show that $\tnf{ran}(\mcl{B}^*)$ is closed.
	
	To show that $\tnf{ran}(\mcl{B}^*)$ is closed, let $(\mbf{y}_n)_{n\in\N}$ be a convergent sequence in $\tnf{ran}(\mcl{B}^*)$ with limit $\mbf{y}$. Then there exists $(\mbf{x}_n)_{n\in\N}$ such that $\mbf{y}_n=\mcl{B}^*\mbf{x}_n$. Now $(\mbf{x}_n)_{n\in\N}$ is Cauchy, since $(\mcl{B}^*\mbf{x}_{n})_{n\in\N}$ is convergent (and therefore Cauchy) and
	\begin{equation*}
		\norm{\mbf{x}_n-\mbf{x}_m}_{H}\leq
		\frac{1}{\epsilon}\norm{\mcl{B}^*\mbf{x}_n-\mcl{B}^*\mbf{x}_m}_{H},\quad \forall m,n\in\N.
	\end{equation*}
	Since $H$ is a Hilbert space, it follows that $(\mbf{x}_n)_{n\in\N}$ converges to some $\mbf{x}\in H$. By continuity of $\mcl{B}^*$ this implies $\mcl{B}^*\mbf{x}=\mbf{y}$, whence $\tnf{ran}(\mcl{B}^*)$ is closed and thus $\tnf{ran}(\mcl{B})=H$.
	
	For necessity, suppose $\tnf{ran}(\mcl{B})=H$. By the open mapping theorem, $\mcl{B}$ is open. Let $U$ denote the open unit ball in $H$. Then there exists $\epsilon>0$ such that $\epsilon U\subseteq\mcl{B}(U)$. It follows that, for all $\mbf{x}\in H$,
	\begin{align*}
		&\norm{\mcl{B}^{*}\mbf{x}}_{H}
		=\sup_{\mbf{y}\in U}|\ip{\mcl{B}^*\mbf{x}}{\mbf{y}}_{H}| \\
		&\qquad=\sup_{\mbf{y}\in U}|\ip{\mbf{x}}{\mcl{B}\mbf{y}}_{H}|
		\geq \sup_{\mbf{z}\in \epsilon U}|\ip{\mbf{x}}{\mbf{z}}_{H}|
		=\epsilon\norm{\mbf{x}}_{H}.
	\end{align*}
	\ \\[-1.65\baselineskip]
\end{proof}

By Lemma~\ref{lem:surjective}, surjectivity of $\lambda\mcl{T}-\mcl{A}\in\mcl{L}(H)$ can be tested by verifying that $\lambda\mcl{T}^*-\mcl{A}^*$ is bounded below. As for the dissipativity constraint, this condition can be posed as an operator inequality and tested using semidefinite programming, as shown in Section~\ref{sec:LPI}.

\subsection{A PIE-Based Generation Theorem for Quasicontraction Semigroups}

Having shown how both conditions of the Lumer--Phillips theorem can be expressed using the PIE representation, $\{\mcl T,\mcl A\}$, we now combine these results to obtain a PIE-based formulation of the Lumer--Phillips theorem. Here, we remark that if $A:D\to H$ generates a $C_{0}$-semigroup on $H$, then for any $\omega\in\R$, $B:=A+\omega I:D\to H$ also generates a $C_{0}$-semigroup, satisfying $e^{tB}=e^{\omega t}e^{tA}$ for all $t\geq 0$~\cite{engel2000C0Semigroups}. Using this observation, and applying Lem.~\ref{lem:dissipative} and Lem.~\ref{lem:surjective}, we obtain the following necessary and sufficient conditions for $\{\mcl{T},\mcl{A}\}$ to generate a quasicontraction semigroup.

\begin{thm}\label{thm:LP_PIE_quasicontraction}
	For a given Hilbert space $H$, let $\mcl{T}:H\to D$ and $\mcl{A}\in\mcl{L}(H)$ be bounded linear operators, with $\mcl{T}$ bijective. Define $A:=\mcl{A}\circ\mcl{T}^{-1}:D\to H$. Then for any $\omega\in\R$, $A$ is the infinitesimal generator of a quasicontraction semigroup with rate $\omega$ on $H$  if and only if the following are satisfied:
	\begin{enumerate}
		\item $\ip{\mcl{T}\mbf{v}}{\mcl{A}\mbf{v}}_{H}\leq \omega\ip{\mcl{T}\mbf{v}}{\mcl{T}\mbf{v}}_{H}$ for all $\mbf{v}\in H$;
		\item there exist $\epsilon>0$ and $\lambda>\omega$ such that $\norm{(\lambda\mcl{T}^*-\mcl{A}^*)\mbf{v}}_{H}\geq\epsilon\norm{\mbf{v}}_{H}$ for all $\mbf{v}\in H$.
	\end{enumerate}
\end{thm}
\begin{proof}
	Define $A_{\omega}:=A-\omega I:D\to H$ and $\mcl{A}_{\omega}:=A_{\omega}\circ\mcl{T}=\mcl{A}-\omega\mcl{T}\in\mcl{L}(H)$. Then for all $\mbf{v}\in H$,
	\begin{equation}\label{eq:dissipative1}
		\ip{\mcl{T}\mbf{v}}{\mcl{A}_{\omega}\mbf{v}}_{H}
		=\ip{\mcl{T}\mbf{v}}{\mcl{A}\mbf{v}}_{H}-\omega\ip{\mcl{T}\mbf{v}}{\mcl{T}\mbf{v}}_{H},	\tag{$*$}
	\end{equation}
	and, for any $\lambda>\omega$ and $\lambda_{\omega}:=\lambda-\omega>0$,
	\begin{equation}\label{eq:surjective1}
		\norm{(\lambda_{\omega}\mcl{T}^*-\mcl{A}_{\omega}^*)\mbf{v}}_{H}
		=\norm{(\lambda\mcl{T}-\mcl{A})^*\mbf{v}}_{H}.		\tag{$\star$}
	\end{equation}
	Now, for sufficiency, suppose both conditions hold. Then, by~\eqref{eq:dissipative1} and Lem.~\ref{lem:dissipative}, it follows that $\ip{\mbf{u}}{A_{\omega}\mbf{u}}_{H}\leq 0$ for all $\mbf{u}\in D$. In addition, by~\eqref{eq:surjective1} and Lem.~\ref{lem:surjective}, $(\lambda_{\omega}\mcl{T}-\mcl{A}_{\omega})\in\mcl{L}(H)$ is surjective for $\lambda_{\omega}>0$. Since $\mcl{T}:H\to D$ is bijective, this implies $(\lambda_{\omega} I-A_{\omega}):=(\lambda_{\omega}\mcl{T}-\mcl{A}_{\omega})\circ\mcl{T}^{-1}:D\to H$ is surjective. By the Lumer--Phillips theorem, it follows that $A_{\omega}$ generates a contraction semigroup on $H$, whence $A$ generates a quasicontraction semigroup with rate $\omega$ on $H$.
	
	Next, for necessity, suppose $A$ generates a quasicontraction semigroup with rate $\omega$, so that $A_{\omega}$ generates a contraction semigroup. By the Lumer--Phillips theorem, this implies $\ip{\mbf{u}}{A_{\omega}\mbf{u}}_{H}\leq 0$ for all $\mbf{u}\in D$, and $\lambda_{\omega} I-A_{\omega}:D\to H$ is surjective for all $\lambda_{\omega}>0$. By Lem.~\ref{lem:dissipative} and~\eqref{eq:dissipative1} it follows that Condition~1) holds. In addition, since $\mcl{T}:H\to D$ is bijective, $\lambda_{\omega} \mcl{T}-\mcl{A}_{\omega}:H\to H$ is surjective for all $\lambda_{\omega}>0$. By Lem.~\ref{lem:surjective} and~\eqref{eq:surjective1}, this implies Condition~2). 
\end{proof}

Theorem~\ref{thm:LP_PIE_quasicontraction} provides necessary and sufficient conditions for $\{\mcl{T},\mcl{A}\}$ to generate a quasicontraction semigroup on $H$, with rate $\omega$. In the following section, we show how these conditions can be formulated as linear operator inequalities, to verify generation of an arbitrary $C_{0}$-semigroup on $L_{2}^{n}[a,b]$.

\section{A PIE-Based Well-Posedness Test for PDEs}\label{sec:LPI}

Using Thm.~\ref{thm:LP_PIE_quasicontraction}, we now derive sufficient conditions under which a PIE-compatible PDE generates a $C_{0}$-semigroup on $L_{2}^{n}[a,b]$, satisfying an exponential bound as defined below.

\begin{defn}\label{defn:rate}
	For a Hilbert space $H$, a $C_{0}$-semigroup $(S(t))_{t\ge0}$ on $H$ has \tbf{growth rate} $\omega\in\R$ and \tbf{gain} $M\ge1$ if\\[-0.75\baselineskip]
	\begin{equation*}
		\norm{S(t)}_{\mcl{L}(H)}\leq Me^{\omega t},\qquad \forall t\geq0 .
	\end{equation*}
\end{defn}
\vspace*{2mm}

It is well-known that any $C_{0}$-semigroup satisfies a bound of this form for some $\omega\in\R$ and $M\geq 1$~\cite{curtain1995introduction}. However, while Thm.~\ref{thm:LP_PIE_quasicontraction} allows arbitrary growth rates $\omega$, it restricts the resulting semigroup to have gain $M=1$. In the following subsection, therefore, we establish conditions for generation of a $C_{0}$-semigroup with arbitrary gain $M$ on $L_{2}^{n}$, by testing generation of a quasicontraction semigroup in the weighted inner product space, $H=L_{2,\mcl{P}}^{n}$. We then show how these conditions can be verified via convex optimization by parameterizing $\mcl{P}$ as a partial integral operator.

\subsection{Generation of $C_{0}$-Semigroups on $L_{2}$}


Thm.~\ref{thm:LP_PIE_quasicontraction} with $H=L_{2}^{n}[a,b]$ provides necessary and sufficient conditions for $\{\mcl{T},\mcl{A}\}$ to generate a quasicontraction semigroup on $L_{2}^{n}$. Now, to verify generation of more general $C_{0}$-semigroups on $L_{2}^{n}$, we note that for any coercive and bounded operator $\mcl{P}$, a $C_{0}$-semigroup on $L_{2,\mcl P}^{n}$ is also a $C_{0}$-semigroup on $L_{2}^{n}$ by equivalence of the inner products. Therefore, sufficient conditions for well-posedness on $L_{2}^{n}$ may be formulated in terms of generation of a quasicontraction semigroup on $H=L_{2,\mcl P}^{n}$, as in the following result.

\begin{thm}\label{thm:LP_PIE_sufficient}
	For $n\in\N$ and $[a,b]\subseteq\R$, let $\mcl{T}:L_{2}^{n}[a,b]\to D$ and $\mcl{A}\in\mcl{L}(L_{2}^{n}[a,b])$ be bounded linear operators, with $\mcl{T}$ bijective. Define $A:=\mcl{A}\circ\mcl{T}^{-1}:D\to L_{2}^{n}[a,b]$. For any $\omega\in\R$, if the following conditions are satisfied,
	\begin{enumerate}
		\item there exist $\epsilon_{1}>0$, $\mcl{P}\in\mcl{L}(L_{2}^{n})$ such that $\mcl{P}\succeq\epsilon_{1}^{2}I$ and
		\begin{equation*}
			\ip{\mcl{T}\mbf{v}}{\mcl{A}\mbf{v}}_{L_{2,\mcl P}}
			\le\omega\ip{\mcl{T}\mbf{v}}{\mcl{T}\mbf{v}}_{L_{2,\mcl P}}, \quad \forall \mbf{v}\in L_{2}^{n};
		\end{equation*}
		
		\item there exist $\epsilon_{2}>0$ and $\lambda>\omega$ such that
		\begin{equation*}
			\norm{(\lambda\mcl{T}^*-\mcl{A}^*)\mbf{v}}_{L_{2}}\ge \epsilon_{2}\norm{\mbf{v}}_{L_{2}},
			\qquad \forall \mbf{v}\in L_{2}^{n};
		\end{equation*}
	\end{enumerate}
	then $A$ generates a $C_{0}$-semigroup on $L_{2}^n$ with growth rate $\omega$ and gain $M:=\frac{\sqrt{\norm{\mcl{P}}_{\mcl{L}(L_{2})}}}{\epsilon_{1}}$.
\end{thm}

\begin{proof}
	Suppose both conditions hold, and define $A_{\omega}:=A-\omega I:D\to L_{2}^{n}$ and $\lambda_{\omega}:=\lambda-\omega>0$. By Condition~1) and Lem.~\ref{lem:dissipative}, we have $	\ip{\mbf{u}}{A_{\omega}\mbf{u}}_{L_{2,\mcl{P}}}\leq 0$ for all $\mbf{u}\in D$.
	Furthermore, by Condition~2) and Lem.~\ref{lem:surjective}, $(\lambda\mcl{T}-\mcl{A})\in\mcl{L}(L_{2}^{n})$ is surjective, and thus $(\lambda I-A):D\to L_{2}^{n}$ is surjective. Since $L_{2}^{n}$ and $L_{2,\mcl{P}}^{n}$ have the same underlying vector space, it follows that $\lambda_{\omega}I-A_{\omega}
	=\lambda I-A:D\to L_{2,\mcl{P}}^{n}$
	is surjective. By the Lumer--Phillips theorem, then, $A_{\omega}$ generates a contraction semigroup on $L_{2,\mcl{P}}^{n}$, and hence $A$ generates a quasicontraction semigroup on $L_{2,\mcl{P}}^{n}$ with growth rate $\omega$. Since the norms on $L_{2}^{n}$ and $L_{2,\mcl{P}}^{n}$ are equivalent, this semigroup is also a $C_{0}$-semigroup on $L_{2}^{n}$.
	Finally, since
	\begin{equation*}
		\epsilon_{1}^2\norm{\mbf{v}}_{L_{2}}^2
		\leq\norm{\mbf{v}}_{L_{2,\mcl P}}^2
		\leq\norm{\mcl{P}}_{\mcl{L}(L_{2})}\norm{\mbf{v}}_{L_{2}}^2,
		\qquad \forall\,\mbf{v}\in L_{2}^{n},
	\end{equation*}
	it follows that for all $\mbf{x}\in L_{2}^{n}$,
	\begin{equation*}
		\norm{e^{tA}\mbf{x}}_{L_{2}}
		\!\leq\!\frac{1}{\epsilon_{1}}\!\norm{e^{tA}\mbf{x}}_{L_{2,\mcl P}}
		\!\leq\!\frac{1}{\epsilon_{1}}e^{\omega t}\norm{\mbf{x}}_{L_{2,\mcl P}}
		\!\leq\! Me^{\omega t}\norm{\mbf{x}}_{L_{2}}.
	\end{equation*}
\end{proof}

Theorem~\ref{thm:LP_PIE_sufficient} provides sufficient conditions for $\{\mcl{T},\mcl{A}\}$ to generate a $C_{0}$-semigroup with gain $M$ and rate $\omega$ on $L_{2}^{n}$. The condition $\ip{\mcl{T}\mbf{v}}{\mcl{A}\mbf{v}}_{\mcl{P}}\leq \omega\ip{\mcl{T}\mbf{v}}{\mcl{T}\mbf{v}}_{\mcl{P}}$ herein is closely related to Lyapunov conditions for exponential stability in the PIE framework, differing only in that well-posedness does not require $\omega\leq 0$. Optimization-based methods for testing this condition have already been proposed for stability analysis in the PIE representation~\cite{shivakumar2024GPDE}. In the following subsection, we show how the well-posedness conditions may be similarly tested using convex optimization methods.

\subsection{An LPI for Well-Posedness of PDEs}\label{subsec:LPI:LPI}

Having established a test for well-posedness of PIEs in Thm.~\ref{thm:LP_PIE_sufficient}, we now show how this result can be applied to PDEs of the form in~\eqref{eq:Cauchy}, on PIE-compatible domains $D$ as in Defn.~\ref{defn:PDE_dom}. For such systems, we can define an associated PIE by $\mcl{T}$ as in Cor.~\ref{cor:Tmap} and $\mcl{A}:=A\circ\mcl{T}$ as in~\eqref{eq:Aop_PIE}. It has been shown that both $\mcl{T}$ and $\mcl{A}$ belong to the class of Partial Integral (PI) operators~\cite{shivakumar2024GPDE}, taking the form
\begin{equation*}
	(\mcl{P}\mbf{v})(s)\!:=\mbs{P}(s)\mbf{v}(s)+\!\int_{a}^{s}\!\!\mbs{L}(s,\theta)\mbf{v}(\theta) d\theta
	+\!\int_{s}^{b}\!\!\mbs{U}(s,\theta)\mbf{v}(\theta) d\theta,
\end{equation*}
and parameterized by polynomials $\mbs{P},\mbs{L},\mbs{U}$. Such PI operators form a class of bounded, linear operators that is closed under adjoint and composition, with explicit formulae for these operations given in~\cite{shivakumar2024GPDE}. In addition, PI operator variables can be parameterized using a finite monomial basis, and positivity and equality constraints can be enforced via semidefinite programming. Using these properties, the following corollary shows how well-posedness in the PIE representation can be verified by solving a convex optimization problem on PI operator variables.

\begin{cor}\label{cor:LPI_wellposed}
	For $n,d\in\N$, $[a,b]\subseteq\R$, and $B,C\in\R^{nd\times nd}$, let the associated domain $D\subseteq W_{2}^{d,n}[a,b]$ as in Defn.~\ref{defn:PDE_dom} be PIE-compatible. For any $\mbs{A}_{k}\in\R^{n\times n}[s]$, define $A:D\to L_{2}^{n}[a,b]$ as in~\eqref{eq:Aop}. Define $\mcl{T}:L_{2}^{n}\to D$ as in Cor.~\ref{cor:Tmap}, and let $\mcl{A}:=A\circ\mcl{T}\in \mcl{L}(L_{2}^{n})$. For any $\epsilon_{1},\epsilon_{2}>0$ and $\omega\in\R$, if there exists a PI operator $\mcl{P}$ and scalar $\lambda>\omega$ satisfying
	\begin{align}\label{eq:LPI_wellposedness}
		\mcl{P}\succeq\epsilon_{1}^2 I_{n},	\qquad
		\mcl{T}^*\mcl{P}\mcl{A}+\mcl{A}^*\mcl{P}\mcl{T}&\preceq 2\omega\mcl{T}^*\mcl{P}\mcl{T},	\\
		(\lambda\mcl{T}-\mcl{A})(\lambda\mcl{T}^*-\mcl{A}^*)&\succeq\epsilon_{2}^2 I_{n},	\notag
	\end{align}
	then $A$ is the infinitesimal generator of a $C_{0}$-semigroup on $L_{2}^{n}[a,b]$ with growth rate $\omega$ and gain $M:=\frac{\sqrt{\norm{\mcl{P}}_{\mcl{L}(L_{2})}}}{\epsilon_{1}}$.
\end{cor}

\begin{proof}
	Suppose $\mcl{P},\omega$ satisfy~\eqref{eq:LPI_wellposedness}. Then, for all $\mbf{v}\in L_{2}^{n}$,
	\begin{align*}
		\ip{\mcl{T}\mbf{v}}{\mcl{A}\mbf{v}}_{L_{2,\mcl P}}
		&=\frac{1}{2}\ip{\mbf{v}}{[\mcl{T}^*\mcl{P}\mcl{A}+\mcl{A}^*\mcl{P}\mcl{T}]\mbf{v}}_{L_{2}}	\\
		&\leq \omega\ip{\mcl{T}\mbf{v}}{\mcl{T}\mbf{v}}_{L_{2,\mcl P}},
	\end{align*}
	as well as
	 $\norm{(\lambda\mcl{T}^*-\mcl{A}^*)\mbf{v}}_{L_{2}}^2\geq \epsilon_{2}^2\norm{\mbf{v}}_{L_{2}}^2$.
	By Thm.~\ref{thm:LP_PIE_sufficient}, it follows that $A$ is the infinitesimal generator of a $C_{0}$-semigroup on $L_{2}^{n}[a,b]$ with growth rate $\omega$ and gain $M$.
\end{proof}

By Cor.~\ref{cor:LPI_wellposed}, well-posedness of a PIE-compatible PDE can be tested by solving the optimization program in~\eqref{eq:LPI_wellposedness}. For each fixed $\omega$, this is a feasibility problem that is linear in the decision variable $\mcl{P}$ and defined by PI operator inequalities. We refer to such a program as a Linear PI Inequality (LPI). Using the PIETOOLS software suite~\cite{shivakumar2025PIETOOLS}, this LPI can be converted to a semidefinite program and solved with an SDP solver such as MOSEK~\cite{mosek}. Using bisection, a smallest value of $\omega$ may then be established for which the program is feasible, thereby certifying well-posedness together with a (numerical) least upper bound on the exponential growth rate. In the next section, we apply this optimization program to several PDE examples and obtain tight upper bounds on the exponential growth rate of solutions.

\section{Illustrations of the Well-Posedness Test}\label{sec:Examples}

Given the well-posedness test for linear PDEs in Cor.~\ref{cor:LPI_wellposed}, we now demonstrate its effectiveness on several classical PDE examples. We then illustrate the utility of the test by applying it to a more complicated, speculative PDE for which well-posedness has not been previously studied. For each example, the optimization program in~\eqref{eq:LPI_wellposedness} is implemented in PIETOOLS~\cite{shivakumar2025PIETOOLS}, with $\epsilon_{1}=\epsilon_{2}=10^{-1}$ and $\lambda=\omega+1$, and the resulting semidefinite program is solved using MOSEK~\cite{mosek}. For each of the classical PDE examples, an ill-posed version is also formulated by modifying either the definition of $A$ or of its domain, and we verify that the LPI in~\eqref{eq:LPI_wellposedness} becomes infeasible. Analytic proofs of well-posedness of the classical PDE examples are provided in Appx.~\ref{appx:PDE_examples}.

\subsection{Transport Equation}\label{subsec:Examples:Transport}

\subsubsection{Well-Posed Example} As a first example, consider the transport equation
\begin{equation}\label{eq:transportPDE}
	u_{t}(t,s)=-u_{s}(t,s),\qquad u(t,0)=0.
\end{equation}
We formulate this PDE as in~\eqref{eq:Cauchy}, by defining $A:=-\partial_{s}:D\to L_{2}[0,1]$ for $D:=\bl\{\mbf{u}\in W_{2}^{1}[0,1]\mid \mbf{u}(0)=0\br\}$.
Then $A$ is the infinitesimal generator of a $C_{0}$-semigroup on $L_{2}$, with exponential growth rate $\omega$ and gain $M=e^{-\omega}$ for any $\omega\leq 0$.
We verify this using Cor.~\ref{cor:LPI_wellposed}, defining the associated PIE representation by $(\mcl{T}\mbf{v})(s):=\int_{0}^{s}\mbf{v}(\theta)d\theta$ and $\mcl{A}:=-I$.
Solving the optimization program in~\eqref{eq:LPI_wellposedness} with $\mcl{P}=I$, well-posedness can be verified with $\omega=0$, certifying that $A$ generates a contraction semigroup.
Furthermore, parameterizing a more general PI operator variable $\mcl{P}$ using monomials of degree at most $2d$, an improved upper bound on the exponential growth rate $\omega$ of solutions is established as in Table~\ref{tab:transport}.

\subsubsection{Ill-Posed Example} Consider now the transport equation with Dirichlet condition at the opposite boundary, so that $A:=-\partial_{s}$ and $D:=\bl\{\mbf{u}\in W_{2}^{1}[0,1]\mid \mbf{u}(1)=0\br\}$, which is ill-posed.
Defining the associated PIE representation---setting $(\mcl{T}\mbf{v})(s)=-\int_s^1 \mbf{v}(\theta)d\theta$ and $\mcl{A}:=-I$---the LPI~\eqref{eq:LPI_wellposedness} is infeasible for all $\omega\in\R$.
We can verify this using $\mbf{v}_{n}\in L_{2}[0,1]$ defined by $\mbf{v}_{n}(s):=n$ for $s\in[0,\frac{1}{n}]$ and $\mbf{v}_{n}(s):=0$ else, for $n\in\N$. 
Then
\begin{equation*}
	(\mcl{T}\mbf{v}_{n})(s)
	=\begin{cases}
		ns-1,	& 	s<\frac{1}{n},\\
		0,	&	\tnf{else},
	\end{cases}\qquad n\in\N,
\end{equation*}
so that $\ip{\mcl{T}\mbf{v}_{n}}{\mcl{T}\mbf{v}_{n}}_{L_{2}}=\frac{1}{3n}$ and $\ip{\mcl{T}\mbf{v}_{n}}{\mbf{v}_{n}}=-\frac{1}{2}$. It follows that, for any $\omega\in\R$ and $\mcl{P}\in \mcl{L}(L_{2})$ satisfying $\mcl{P}\succeq \epsilon_{1}^{2}I$, letting $n\to\infty$,
\begin{equation*}
	\frac{\ip{\mcl{T}\mbf{v}_{n}}{\mcl{A}\mbf{v}_{n}}_{L_{2,\mcl{P}}}}{\omega\ip{\mcl{T}\mbf{v}_{n}}{\mcl{T}\mbf{v}_{n}}_{L_{2,\mcl{P}}}}	
	\geq \frac{\epsilon_{1}^2}{\norm{\mcl{P}}_{\mcl{L}(L_{2})}\omega} \frac{-\ip{\mcl{T}\mbf{v}_{n}}{\mbf{v}_{n}}_{L_{2}}}{\ip{\mcl{T}\mbf{v}_{n}}{\mcl{T}\mbf{v}_{n}}_{L_{2}}}\to\infty,
\end{equation*}
whence Condition 1) of Thm.~\ref{thm:LP_PIE_sufficient} cannot be satisfied.

\begin{table}[t]
	\setlength{\tabcolsep}{3.2pt}
	\begin{tabular}{c|ccccccccc}
		$d$ & 0 & 1 & 2 & 3 & 4 & 5 & 6 & 7 & 8\\\hline
		$\omega$ & 0 & -0.999 & -2.366 & -3.874 & -4.834 & -4.958 & -5.067 & -5.079 & -5.100
	\end{tabular}
	\caption{Least upper bound on the exponential growth rate $\omega$ with which well-posedness of the transport equation in~\eqref{eq:transportPDE} can be verified by solving the LPI in Cor.~\ref{cor:LPI_wellposed}, parameterizing $\mcl{P}$ by monomials of degree at most $2d$.}
	\label{tab:transport}
	\vspace*{-0.4cm}
\end{table}

\subsection{Heat Equation}\label{subs:Examples:Heat}

\subsubsection{Well-Posed Example} Next, consider the following reaction–diffusion equation,
\begin{align*}
	u_t(t,s)&=u_{ss}(t,s)+ru(t,s), & t\ge0,\ s\in[0,1],\\
	u(t,0)&=u(t,1)=0,
\end{align*}
where $r\in\R$. Representing this PDE as in~\eqref{eq:Cauchy} with
\begin{equation*}
A:=\partial_s^2+rI,\quad
D=\bl\{\mbf{u}\in W_2^2[0,1]\mid \mbf{u}(0)=\mbf{u}(1)=0\br\},
\end{equation*}
$A$ generates a $C_0$-semigroup with exponential growth rate $\omega=r-\pi^2$ and gain $M=1$ on $L_{2}[0,1]$, for all $r\in\R$. We can represent this system as a PIE by defining
\begin{equation}\label{eq:Top_heatPDE}
	(\mcl{T}\mbf{v})(s)=
	\int_0^s \theta(s-1)\mbf{v}(\theta)d\theta
	+\int_s^1 s(\theta-1)\mbf{v}(\theta)d\theta,
\end{equation}
and $\mcl{A}:=A\circ\mcl{T}=I+r\mcl{T}$. Solving the LPI in~\eqref{eq:LPI_wellposedness}, parameterizing $\mcl{P}$ by polynomials of degree at most $6$, well-posedness can be verified for any $r\in\{0,\dots,20\}$, with least upper bound $\omega=r-\pi^2+10^{-5}$ on the growth rate.

\subsubsection{Ill-Posed Example} Consider now an ill-posed, reverse-time heat equation, defined by
\begin{equation*}
	A:=-\partial_s^2,\quad
	D=\bl\{\mbf{u}\in W_2^2[0,1]\mid \mbf{u}(0)=\mbf{u}(1)=0\br\}.
\end{equation*}
The associated PIE representation is defined by $\mcl{T}$ as in~\eqref{eq:Top_heatPDE} and $\mcl{A}=-I$.
Considering then $\mbf{v}_{n}(s):=-n\pi\sin(n\pi s)$, we have for each $n\in\N$ that $(\mcl{T}\mbf{v}_{n})(s)=\frac{1}{n\pi}\sin(n\pi s)$ and $\mcl{A}\mbf{v}_{n}=-\mbf{v}_{n}=n^2\pi^2\mcl{T}\mbf{v}_{n}$. It follows that
\begin{equation*}
	\ip{\mcl{T}\mbf{v}_{n}}{\mcl{A}\mbf{v}_{n}}_{L_{2,\mcl{P}}}
	=n^2\pi^2\ip{\mcl{T}\mbf{v}_{n}}{\mcl{T}\mbf{v}_{n}}_{L_{2,\mcl{P}}},\quad \forall n\in\N.
\end{equation*}
Thus, for any $\omega\in\R$ and coercive and bounded $\mcl{P}$, there exists $n\in\N$ such that $\ip{\mcl{T}\mbf{v}_{n}}{\mcl{A}\mbf{v}_{n}}_{L_{2,\mcl{P}}}>\omega \ip{\mcl{T}\mbf{v}_{n}}{\mcl{T}\mbf{v}_{n}}_{L_{2,\mcl{P}}}$.
We find that Condition 1) of Thm.~\ref{thm:LP_PIE_sufficient} cannot be satisfied, and the well-posedness LPI~\eqref{eq:LPI_wellposedness} is infeasible.

\subsection{Wave Equation}\label{subsec:Examples:Wave}

\subsubsection{Well-Posed Example} Now, consider the wave equation
\begin{align}\label{eq:wave}
	u_{tt}(t,s)&=u_{ss}(t,s), & t\ge0,\ s\in[0,1],\\
	u_t(t,0)&=u_s(t,1)=0.		\notag
\end{align}
To allow for well-posedness (on $L_{2}^{2}$), we consider a representation in terms of $\mbf{u}(t)=(\mbf u_1(t),\mbf u_2(t))=(u_t(t),u_s(t))$, expressing the PDE dynamics of $\mbf{u}$ as in~\eqref{eq:Cauchy} with
\begin{equation*}
	A=\bmat{0&\partial_s\\\partial_s&0},\quad
	D=\bbl\{\mbf{u}\in W_2^{1,2}[0,1]\,\bbl|\ \mat{\mbf{u}_1(0)=0\\ \mbf{u}_2(1)=0}\,\bbr\}.
\end{equation*}
Then, $A$ generates a contraction semigroup on $L_{2}^{2}[0,1]$.
For this system, the equivalent PIE is defined by
\begin{equation*}
(\mcl{T}\mbf{v})(s)=
\int_0^s
\bmat{1&0\\0&0}\mbf{v}(\theta)d\theta
-\int_s^1
\bmat{0&0\\0&1}\mbf{v}(\theta)d\theta,
\end{equation*}
with $\mcl{T}^{-1}=\partial_s$, and $\mcl{A}=A\circ\mcl{T}=\sbrray{cc}{0&1\\1&0}$. The optimization program~\eqref{eq:LPI_wellposedness} is feasible with $\omega=0$ and $\mcl{P}=I$, certifying that $A$ generates a contraction semigroup.

\subsubsection{Ill-Posed Example} For the same wave equation in~\eqref{eq:wave}, suppose we instead define the state $\mbf{u}(t)=(\mbf u_1(t),\mbf u_2(t))=(u(t),u_t(t))$ with
\begin{equation*}
A=\bmat{0&I\\\partial_s^2&0},\quad
D=\bbl\{\mbf{u}\in W_2^{2,2}[0,1]\,\bbl|\ \mat{\mbf{u}(0)=0\\ \mbf{u}_s(1)=0}\,\bbr\}.
\end{equation*}
Then, the resulting PDE in~\eqref{eq:Cauchy} is not well-posed, and LPI~\eqref{eq:LPI_wellposedness} is infeasible for the associated PIE representation. In particular, computing $\mcl{T}=(\partial_s^2)^{-1}$ and $\mcl{A}=A\circ\mcl{T}$, the resulting operator $\mcl{Q}:=(\lambda\mcl{T}-\mcl{A})(\lambda\mcl{T}-\mcl{A})^*$ for any $\lambda\in\R$ will be of the form
\begin{equation*}
	(\mcl{Q}\mbf{v})(s)=\bmat{0&0\\0&1}\mbf{v} +\int_{0}^{1}\mbs{K}(s,\theta)\mbf{v}(\theta)d\theta,
\end{equation*}
for some (semi-separable) kernel $\mbs{K}$. On the infinite-dimensional subspace where $\mbf{v}_{2}=0$, this operator $\mcl{Q}$ is compact, and therefore cannot satisfy $\mcl{Q}\succeq \epsilon I$ for any $\epsilon>0$. Thus, Condition 2) in Thm.~\ref{thm:LP_PIE_sufficient} is infeasible.

\subsection{Coupled PDEs}

To illustrate the application of Cor.~\ref{cor:LPI_wellposed} for verifying well-posedness of less classically well-studied PDEs, consider the following, entirely speculative, system of coupled PDEs
\begin{align*}
	\partial_{t}u_{1}(t,s)&=-\partial_{s}u_{1}(t,s)+s\partial_{t}u_{3}(t,1),	\\
	\partial_{t}u_{2}(t,s)&=\partial_{s}^{2}u_{2}(t,s)+s(2-s)u_{1}(t,1)+\partial_{s}u_{3}(t,0),	\\
	\partial_{t}^{2}u_{3}(t,s)&=\partial_{s}^{2}u_{3}(t,s)-\partial_{s}u_{2}(t,s),	\\
	u_{1}(t,0)&=u_{2}(t,1)=\partial_{t}u_{3}(t,0),\qquad u_{2}(t,0)=0,\\
	\partial_{s}u_{2}(t,1)&=\partial_{s}u_{3}(t,0),\hspace*{2.0cm} \partial_{s}u_{3}(t,1)=0.
\end{align*}
Introducing the state $\mbf{u}(t)=(u_{1}(t),u_{2}(t),\partial_{t}u_{3}(t),\partial_{s}u_{3}(t))$, this system can be formulated as in~\eqref{eq:Cauchy} using a suitable operator $A:D\to L_{2}^{4}[0,1]$ on the domain
\begin{equation*}
	D:=\bbbl\{\mbf{u}\in \slbmat{W_{2}^{1}[0,1]\\W_{2}^{2}[0,1]\\W_{2}^{1,2}[0,1]}\,\bbl|\ \slmat{\mbf{u}_{1}(0)=\mbf{u}_{2}(1)=\mbf{u}_{3}(0)\\ \mbf{u}_{2}(0)=\mbf{u}_{4}(1)=0\\ \partial_{s}\mbf{u}_{2}(1)=\mbf{u}_{4}(0)}\bbbr\}.
\end{equation*}
An explicit expression for $A$ and the operators $\{\mcl{T},\mcl{A}\}$ defining the associated PIE representation are omitted for brevity. Solving the LPI in~\eqref{eq:LPI_wellposedness}, we can verify that this PDE is well-posed, generating a $C_{0}$-semigroup with exponential growth rate $\omega=0$.

\section{Conclusion}

A PIE-based framework for verifying well-posedness of partial differential equations was developed. By reformulating the Lumer–Phillips generation theorem in terms of Partial Integral (PI) operators, necessary and sufficient conditions were established for a PIE to generate a quasicontraction semigroup. These conditions were relaxed to determine sufficient conditions for generation of a more general $C_{0}$-semigroup, with an upper bound on the exponential growth rate of solutions. Using these results, a convex optimization program on PI operator variables was developed for testing well-posedness of any PDE admitting an equivalent PIE representation. The resulting conditions can be implemented using PIETOOLS and solved using semidefinite programming. Illustrative examples demonstrated that the method can verify well-posedness and accurately estimate growth rates for several 1D PDE systems.

\bibliographystyle{IEEEtran}
\bibliography{bibfile}
%
%


\bigskip

\begin{appendices}
	
\section{Well-Posedness of Classical PDE Examples}\label{appx:PDE_examples}

In this appendix, we consider the classical PDE examples from Sec.~\ref{sec:Examples}, and verify well-posedness properties. In each case, we will demonstrate well-posedness by explicitly defining the $C_{0}$-semigroup generated by the PDE, or proving that such a $C_{0}$-semigroup does not exist.

\subsection{Transport Equation}

Consider first $A:=-\partial_{s}:D\to L_{2}[0,1]$ on
\begin{equation*}
	D:=\{\mbf{u}\in W_{2}^{1}[0,1]\mid \mbf{u}(0)=0\},
\end{equation*}
defining a transport equation with Dirichlet boundary condition at the lower boundary. 
A solution to the transport equation can be readily found using the method of characteristics, noting that $u_{t}(t,s)=-u_{s}(t,s)$ and $u(0,s)=\mbf{x}(s)$ imply $u(t,s)=\mbf{x}(s-t)$ for all $t\leq s$. Imposing $u(t,0)=0$, it follows that $A$ generates the $C_{0}$-semigroup
\begin{equation*}
	(e^{tA}\mbf{x})(s):=\begin{cases}
		\mbf{x}(s-t),	&	t\leq s,	\\
		0,		&	t>s.
	\end{cases}
\end{equation*}
Here, $\norm{e^{tA}\mbf{x}}_{L_{2}}\leq \norm{e^{\tau A}\mbf{x}}_{L_{2}}$ for all $t\geq \tau\geq 0$ and $\mbf{x}\in D$, whence $(e^{tA})_{t\geq 0}$ is a contraction semigroup. Furthermore since $\norm{e^{tA}\mbf{x}}_{L_{2}}=0$ for all $t\geq 1$ and $\mbf{x}\in D$, it follows that $\norm{e^{tA}\mbf{x}}_{L_{2}}\leq e^{\omega [t-1]}\norm{\mbf{x}}_{L_{2}}$ for any $\omega\in\R$ and all $t\geq 1$ and $\mbf{x}\in D$. Thus, the $C_{0}$-semigroup generated by $A$ has exponential growth rate $\omega$ for any $\omega\in\R$, with gain $M=e^{-\omega}$. 

However, consider now $A:=-\partial_{s}:\hat{D}\to L_{2}[0,1]$ on $\hat{D}:=\{\mbf{u}\in W_{2}^{1}[0,1]\mid \mbf{u}(1)=0\}$,
imposing a Dirichlet boundary condition on the upper boundary. In this case, a $C_{0}$-semigroup generated by $A$ would still need to satisfy $(e^{tA}\mbf{x})(s)=\mbf{x}(s-t)$ for $t\leq s$. However, we may choose $\mbf{x}\in \hat{D}$ and $t\in(0,1)$ such that $\mbf{x}(1-t)\neq 0$. Then $(e^{tA}\mbf{x})(1)=\mbf{x}(1-t)\neq 0$, implying that $e^{tA}\mbf{x}\notin \hat{D}$. 
Therefore, the associated transport equation is ill-posed.

\subsection{Heat Equation}

Next, consider $A:=\partial_{s}^2:D\to L_{2}[0,1]$ on
\begin{equation*}
	D:=\{\mbf{u}\in W_{2}^{2}[0,1]\mid \mbf{u}(0)=\mbf{u}(1)=0\},
\end{equation*}
defining the heat equation with Dirichlet boundary conditions. The unique solution to the heat equation can be readily established using separation of variables, finding that $A$ generates the $C_{0}$-semigroup with growth rate $\omega=-\pi^2$,
\begin{align*}
	(e^{tA}\mbf{x})(s)
	=\sum_{n=1}^{\infty}2\ip{\mbf{x}}{\mbs{\phi}_{n}}_{L_{2}}e^{-n^2\pi^2t}\mbs{\phi}_{n}(s),
\end{align*}
where $\mbs{\phi}_{n}(s):=\sin(n\pi s)$ for $n\in\N$. Consider now the reverse-time heat equation, defined by $\hat{A}:=-\partial_{s}^2:D\to L_{2}[0,1]$, on the same domain $D$. Then the $C_{0}$-semigroup generated by $\hat{A}$ would have to satisfy $e^{t\hat{A}}=e^{-tA}$.
However, $e^{-tA}\notin\mcl{L}(L_{2})$ for any $t>0$, since for any $M>0$, we can establish $n\in\N$ sufficiently large such that $e^{n^2\pi^2 t} >2M$, whence $\norm{e^{t\hat{A}}\mbs{\phi}_{n}}_{L_{2}}>M$. Thus, $\hat{A}:D\to L_{2}[0,1]$ does not generate a $C_{0}$-semigroup, implying the reverse-time heat equation is not well-posed.

\subsection{Wave Equation}

Consider the following wave equation,
\begin{align*}
	u_{tt}(t,s)&=u_{ss}(t,s),\qquad t\geq 0,~s\in[0,1],\\
	u(t,0)&=u_{s}(t,1)=0
\end{align*}
Solutions to this PDE are given by
\begin{align*}
	u(t,s)&:=\sum_{n=0}^{\infty}\bbbl(a_{n}\cos(\mu_{n}t)+\frac{b_{n}}{\mu_{n}}\sin(\mu_{n}t)\bbbr)\mbs{\phi}_{n}(s),
\end{align*}
where $\mu_{n}:=[n+\frac{1}{2}]\pi$ and $\mbs{\phi}_{n}(s):=\sin(\mu_{n}s)$ for $n\in\N$, and where
\begin{equation*}
	a_{n}:=2\ip{\mbs{\phi}_{n}}{u(0,\cdot)}_{L_{2}} \qquad
	b_{n}:=2\ip{\mbs{\phi}_{n}}{u_{t}(0,\cdot)}_{L_{2}}.
\end{equation*}
Letting $\mbs{\psi}_{n}(s):=\cos(\mu_{n}(s))$, it follows that
\begin{align*}
	u_{t}(t,s)&:=\sum_{n=0}^{\infty}\bl(-a_{n}\mu_{n}\sin(\mu_{n}t)+b_{n}\cos(\mu_{n}t)\br)\mbs{\phi}_{n}(s),	\\
	u_{s}(t,s)&:=\sum_{n=0}^{\infty}\bl(a_{n}\mu_{n}\cos(\mu_{n}t)+b_{n}\sin(\mu_{n}t)\br)\mbs{\psi}_{n}(s).
\end{align*} 

Now, to formulate the wave equation as in~\eqref{eq:Cauchy}, we may introduce the state $\mbf{u}(t)=\bmat{u_{t}(t)\\u_{s}(t)}$, defining
\begin{align*}
	&A:=\bmat{0&\partial_{s}\\\partial_{s}&0}:D\to L_{2}^{2}[0,1],	\\
	&D:=\{(\mbf{u}_{1},\mbf{u}_{2})\in W_{2}^{1,2}[0,1]\mid \mbf{u}_{1}(0)=\mbf{u}_{2}(1)=0\}.
\end{align*}
Then, the $C_{0}$-semigroup generated by $A$ is given by
\begin{align*}
	(e^{tA}\mbf{x})(s):=
	\sum_{n=0}^{\infty}\bmat{\bl(-a_{n}\mu_{n}\sin(\mu_{n}t)+b_{n}\cos(\mu_{n}t)\br)\mbs{\phi}_{n}(s)\\ \bl(a_{n}\mu_{n}\cos(\mu_{n}t)+b_{n}\sin(\mu_{n}t)\br)\mbs{\psi}_{n}(s)},
\end{align*}
where now $a_{n}:=\frac{2}{\mu_{n}}\ip{\mbs{\psi}_{n}}{\mbf{x}_{2}}_{L_{2}}$ and $b_{n}:=2\ip{\mbs{\phi}_{n}}{\mbf{x}_{1}}_{L_{2}}$ for $\mbf{x}=\bmat{\mbf{x}_{1}\\\mbf{x}_{2}}\in D$. 
By orthogonality of the basis functions, it follows that
\begin{align*}
	\norm{e^{tA}\mbf{x}}_{L_{2}}^2
	&=\sum_{n=0}^{\infty} \frac{1}{2} \bbl[ \bl(-a_{n}\mu_{n}\sin(\mu_{n}t)+b_{n}\cos(\mu_{n}t)\br)^2	\\[-0.4em]
	&\hspace*{1.75cm}+\bl(a_{n}\mu_{n}\cos(\mu_{n}t)+b_{n}\sin(\mu_{n}t)\br)^2\bbr]	\\
	&=\frac{1}{2}\sum_{n=0}^{\infty}  \bl[a_{n}^2\mu_{n}^2+b_{n}^2 \br]
	=\norm{\mbf{x}}_{L_{2}}^{2}.
\end{align*}
Hence, $e^{tA}$ is a contraction semigroup.

Alternatively, the wave equation can also be represented as in~\eqref{eq:Cauchy} by introducing $\mbf{u}(t)=\bmat{u(t)\\u_{t}(t)}$ and defining
\begin{align*}
	&A:=\bmat{0&I\\\partial_{s}^2&0}:D\to L_{2}^{2}[0,1],\\
	&D:=\bbbl\{\mbf{u}\in W_{2}^{2,2}[0,1]\,\bbl|\ \mat{\mbf{u}_{1}(0)=\partial_{s}\mbf{u}_{1}(1)=0\\\mbf{u}_{2}(0)=\partial_{s}\mbf{u}_{2}(1)=0}\bbbr\}.
\end{align*}
Then, if $A$ is the generator of a $C_{0}$-semigroup, this semigroup must satisfy
\begin{align*}
	(e^{tA}\mbf{x})(s):=
	\sum_{n=0}^{\infty}\bmat{a_{n}\cos(\mu_{n}t)+\frac{b_{n}}{\mu_{n}}\sin(\mu_{n}t)\\ -a_{n}\mu_{n}\sin(\mu_{n}t)+b_{n}\cos(\mu_{n}t)}\mbs{\phi}_{n}(s),
\end{align*}
where now $a_{n}:=2\ip{\mbs{\phi}_{n}}{\mbf{x}_{1}}_{L_{2}}$ and $b_{n}:=2\ip{\mbs{\phi}_{n}}{\mbf{x}_{2}}_{L_{2}}$ for $\mbf{x}=\bmat{\mbf{x}_{1}\\\mbf{x}_{2}}\in D$. 
By orthogonality of the basis functions, it follows that
\begin{align*}
	\norm{e^{tA}\mbf{x}}_{L_{2}}^2
	&=\sum_{n=0}^{\infty} \frac{1}{2} \bbbl[\bbbl(a_{n}\cos(\mu_{n}t)+\frac{b_{n}}{\mu_{n}}\sin(\mu_{n}t)\bbbr)^2 	\\[-0.4em]
	&\hspace*{1.5cm}+\bl(-a_{n}\mu_{n}\sin(\mu_{n}t)+b_{n}\cos(\mu_{n}t)\br)^2\bbbr].	
\end{align*}
However, consider now $\mbf{x}(s)=\bmat{\sin(\mu_{m}s)\\0}$ for some $m\in\N$, so that $b_{n}:=0$ for all $n\in\N$, and $a_{n}=0$ for all $n\in\N\setminus\{m\}$, with $a_{m}=1$. Then
\begin{align*}
	\norm{e^{tA}\mbf{x}}_{L_{2}}^2
	&=\frac{1}{2} \bbl[\cos^2(\mu_{m}t) +\mu_{m}^2\sin^2(\mu_{m}t)\bbr],
\end{align*}
whence $\norm{\mbf{x}}_{L_{2}}^2=\frac{1}{2}$ and $\norm{e^{tA}\mbf{x}}_{L_{2}}^2|_{t=1}=\frac{1}{2}\mu_{m}^2$. It follows that, for any $M\geq 1$ and $\omega\in\R$, we can find sufficiently large $m\in\N$ such that \[
\norm{e^{tA}\mbf{x}}_{L_{2}}^2|_{t=1}=\frac{1}{2}\mu_{m}^2>M^2e^{2\omega}\norm{\mbf{x}}_{L_{2}}^2.
\]
However, if $A$ generated a $C_{0}$-semigroup, there would exist $M\geq 1$ and $\omega\in\R$ such that
\[
\norm{e^{tA}\mbf{x}}_{L_{2}}\leq Me^{\omega t}\norm{\mbf{x}}_{L_{2}},\quad \forall t\geq 0.
\]
This contradicts the preceding inequality. Therefore, the considered operator $A:D\to L_{2}^{2}$ cannot be the infinitesimal generator of a $C_{0}$-semigroup, and hence this formulation of the wave equation is not well-posed.

\end{appendices}

\end{document}